\documentclass[12pt]{amsart}

\usepackage{amsthm,amssymb,color,hyperref,mathtools,tikz}
\usepackage[margin=1in]{geometry}

\theoremstyle{definition}
\newtheorem{thm}{Theorem}[section]

\newtheorem{prop}[thm]{Proposition}

\newtheorem{example}[thm]{Example}
\newtheorem{conj}[thm]{Conjecture}

\newtheorem{defn}[thm]{Definition}
\newtheorem{cor}[thm]{Corollary}

\newtheorem{question}[thm]{Question}

\DeclareMathOperator{\Av}{Av}
\DeclareMathOperator{\Des}{Des}

\DeclareMathOperator{\des}{des}
\DeclareMathOperator{\inv}{inv}
\DeclareMathOperator{\maj}{maj}
\DeclareMathOperator{\Exc}{Exc}
\DeclareMathOperator{\exc}{exc}
\DeclareMathOperator{\Inv}{Inv}
\DeclareMathOperator{\lcm}{lcm}
\DeclareMathOperator{\st}{st}
\DeclareMathOperator{\std}{std}

\DeclarePairedDelimiter{\multiset}{\{\!\{}{\}\!\}}

\newcommand{\NN}{\mathbb{N}}

\newcommand{\ZZ}{\mathbb{Z}}
\newcommand{\symm}{\mathfrak{S}}

\begin{document}

\title{Width-$k$ Generalizations of Classical Permutation Statistics}
\author{Robert Davis}

\begin{abstract}
	We introduce new natural generalizations of the classical descent and inversion statistics for permutations, called \emph{width-$k$ descents} and \emph{width-$k$ inversions}.
	These variations induce generalizations of the excedance and major statistics, providing a framework in which the most well-known equidistributivity results for classical statistics are paralleled.
	We explore additional relationships among the statistics providing specific formulas in certain special cases.
	Moreover, we explore the behavior of these width-$k$ statistics in the context of pattern avoidance. 
\end{abstract}

\thanks{Department of Mathematics, Michigan State University, East Lansing, MI 48824-1027, USA. E-mail: \href{mailto:davisr@math.msu.edu}{\nolinkurl{davisr@math.msu.edu}}}
\maketitle

\section{Introduction}

Let $\symm_n$ denote the set of permutations $\sigma = a_1\cdots a_n$ of $[n]$, and let $\symm = \symm_1 \cup \symm_2 \cup \dots$.
A function $\st: \symm_n \to \NN$ is called a \emph{statistic}, and the systematic study of permutation statistics is generally accepted to have begun with MacMahon \cite{MacMahonCA}. 
In particular, four of the most well-known statistics are the \emph{descent}, \emph{inversion}, \emph{major}, and \emph{excedance} statistics, defined respectively by
\begin{eqnarray*}
	\des \sigma &=& | \{i \in [n-1]\ |\ a_i > a_{i+1} \}| \\
	\inv \sigma &=& | \{(i,j) \in [n]^2\ |\ i < j \text{ and } a_i > a_j \}| \\
	\maj \sigma &=& \sum_{i \in \Des \sigma} i \\
	\exc \sigma &=& |\{ i \in [n]\ |\ a_i > i\}|
\end{eqnarray*}
where $\Des \sigma = \{i \in [n-1]\ |\ a_i > a_{i+1}\}$.

Given any statistic $\st$, one may form the generating function
\[
	F_n^{\st}(q) = \sum_{\sigma \in \symm_n} q^{\st \sigma}.
\]
A famous result due to MacMahon \cite{MacMahonCA} states that $F_n^{\des}(q) = F_n^{\exc}(q)$, and that both equal are to the \emph{Eulerian polynomial} $A_n(q)$.
The Eulerian polynomials themselves may be defined via the identity
\[
	\sum_{j \geq 0} (1+j)^nq^j = \frac{A_n(q)}{(1-q)^{n+1}}.
\]
Another well-known result, due to MacMahon \cite{MacMahonMaj} and Rodrigues \cite{RodriguesInv}, states $F_n^{\inv}(q) = F_n^{\maj}(q) = [n]_q!$
Foata famously provided combinatorial proofs of the equidistributivity of each; see \cite{Rearrangements} for a treatment of these.
Thus any statistic $\st$ for which $F_n^{\st}(q) = A_n(q)$ is called \emph{Eulerian}, and if $F_n^{\st}(q) = [n]_q!$ then $\st$ is called \emph{Mahonian}.
These four statistics have many generalizations; in this article, we discuss new variations, induced from a simple generalization of $\des$.

For each of the following definitions, we assume $n \in \ZZ_{>0}$, $k \in [n-1]$, $\emptyset \neq K \subseteq [n-1]$, and $\sigma = a_1a_2\ldots a_n \in \symm_n$. 
We define a \emph{width-$k$ descent} of $\sigma$ to be an index $i \subseteq [n-k]$ for which $a_i > a_{i+k}$.
Thus the width-$1$ descents are the usual descents of a permutation.
Denote the set of all width-$k$ descents of $\sigma$ by 
\[
	\Des_k(\sigma) = \{i \in [n-k] \mid a_i > a_{i+k}\},
\]
and set 
\[
	\des_k(\sigma) = |\Des_k(\sigma)|.
\]

If one is interested in descents of $\sigma$ of various widths, first let $K \subseteq [n-1]$ denote the set of widths under consideration.
Now, we define $\Des_K(\sigma)$ to be the multiset $\bigcup_{k\in K} \Des_k(\sigma)$, and $\des_K(\sigma) = |\Des_K(\sigma)|$.

Now, define a \emph{width-$k$ inversion} of $\sigma$ to be a pair $(i,j)$ for which $a_i > a_j$ and $j - i = mk$ for some positive integer $m$.
Denote the set of width-$k$ inversions by
\[
	\Inv_k(\sigma) = \{ (i,j) \in [n]^2 \mid a_i > a_j \text{ and } j - k = mk, k \in \ZZ\}
\]
and set
\[
	\inv_k(\sigma) = |\Inv_k(\sigma)|.
\]
Again, one may be interested in width-$k$ inversions for multiple values of $k$, so for $K \subseteq [n-1]$ define $\Inv_K(\sigma) = \bigcup_{k \in K} \Inv_k(\sigma)$.
Additionally, let $\inv_K(\sigma) = |\Inv_K(\sigma)|$.

\begin{example}\label{ex:first example}
	If $\sigma = 4136572$, then
	\[
		\Des_{\{2,3\}}(\sigma) = \{1,4,5\} \text{ and } \Inv_{\{2,3\}}(\sigma) = \{(1,3),(1,7),(3,7),(4,7),(5,7)\}.
	\]
	Thus, $\des_{\{2,3\}}(\sigma) = 3$, $\inv_{\{2,3\}}(\sigma) = 5$.
\end{example}

As we will see in the next section, the above definitions motivate us to generalize $\exc$ and $\maj$ in such a way that well-known known relationships among $\des,\inv,\maj,\exc$ are paralleled.
However, it is not very convenient to work with $\exc_K$ and $\maj_K$ directly.
So, this paper will focus much more on $\des_K$ and $\inv_K$.

The main focus of this paper is to explore these new statistics and their relationships among each other. 
While they are well-behaved for special cases of $K \subseteq [n-1]$, formulas for more general cases have been more elusive.
Section~\ref{sec:avoidance} continues this exploration by considering the same statistics for classes of permutations avoiding a variety of patterns.


\section{Main Results}

We begin with simple expressions for $\des_K(\sigma)$ and $\inv_K(\sigma)$ for a fixed $\sigma \in \symm_n$.
First, we point out that if $K \subseteq [n-1]$, $j \in K$ and $ij \in K$ for some positive integer $i$, then for any $\sigma \in \symm_n$, $\inv_K(\sigma) = \inv_{K \setminus \{ij\}}(\sigma)$.
This occurs because for any $k \in [n-1]$, $\inv_k(\sigma)$ counts all descents whose widths are multiples of $k$.
Thus $ij$ is already accounted for in $\inv_{K \setminus \{j\}}(\sigma)$.
This follows quickly from the definitions of $\inv_k$ and $\des_k$.

\begin{prop}
	For any nonempty $K \subseteq [n-1]$,
	\[
		\inv_K(\sigma) = \sum_{\emptyset \subsetneq K' \subseteq K} (-1)^{|K'| + 1}\inv_{\lcm(K')}(\sigma),
	\]
	where we set $\inv_{\lcm(K')}(\sigma) = 0$ if $\lcm(K') \geq n$.
\end{prop}

\begin{proof}
	We first consider the case where $K = \{k\}$.
	The elements of $\Inv_k(\sigma)$ are pairs of the form $(i,i+jk)$ for some positive integer $j$.
	Such an element exists if and only if there is a width-$(jk)$ descent of $\sigma$ at $i$.
	Thus, $\inv_k(\sigma)$ simply counts the number of width-$jk$ descents of $\sigma$ for all possible $j$.
	This leads to the equality
	\[
		\inv_k(\sigma) = \sum_{j \geq 1} \des_{jk}(\sigma).
	\]	
	The formula for general $K$ then follows from inclusion-exclusion:
	by adding the number of all width-$k$ inversions for all $k \in K$, we are twice counting any instances of a width-$\lcm(k_1,k_2)$ inversion since such an inversion is also of widths $k_1$ and $k_2$.
	If $K$ contains three distinct elements $k_1,k_2,k_3$, then $\lcm(k_1,k_2,k_3)$ would have been added three times (once for each $\inv_{k_i}(\sigma)$) and subtracted three times (once for each $\inv_{\{k_i,k_j\}}(\sigma)$, $i < j$), so it must be added again for the sum.
	Extending this argument to range over larger subsets of $K$ results in the claimed formula.
\end{proof}

\begin{example}
	Let us return to $\sigma = 4136572$. We saw from Example~\ref{ex:first example} that $\inv_{\{2,3\}}(\sigma) = 5$,
	where four inversions have width $2$ and two have width $3$.
	But since the inversion $(1,7)$ has width both $2$ and $3$, it must also have width $\lcm(2,3) = 6$.
	So, it contributes two summands of $1$ and one summand of $-1$.
\end{example}

We come now to a function which helps demonstrate the interactions among the width-$k$ statistics.
Let $n$ and $k$ be positive integers for which $n = dk + r$ for some $d,r \in \ZZ$ with $0 \leq r < k$.
To each $\sigma = a_1\dots a_n \in \symm_n$ we may then associate the set of disjoint substrings
$\beta_{n,k}(\sigma) = \{\beta_{n,k}^1(\sigma),\dots, \beta_{n,k}^k(\sigma)\}$ where
\[
	\beta_{n,k}^i(\sigma) = \begin{cases}
						a_ia_{i+k}a_{i+2k}\dots a_{i + dk} & \text{ if } i \leq r\\
						a_ia_{i+k}a_{i+2k}\dots a_{i + (d-1)k} & \text{ if } r < i < k
					\end{cases}
\]
Now, define
\[
	\phi: \symm_n \to \symm_{d+1}^r \times \symm_d^{k-r}
\]
by setting $\phi(\sigma) = (\std \beta_{n,k}^1(\sigma), \dots, \std \beta_{n,k}^k(\sigma))$, where $\std$ is the \emph{standardization} map, that is,  the permutation obtained by replacing the 
smallest element of $\sigma$ with $1$, the second-smallest element with $2$, etc.
Note in particular that each $\std \beta_{n,k}^i(\sigma)$ is a permutation of $[d+1]$ or $[d]$.

The first of the identities in the following proposition was originally established in \cite{SackUlfarsson}, though with slightly different notation. 
We provide an alternate proof and extend their result to width-$k$ inversions.
 
\begin{thm}\label{thm:main bijection}
	Let $n$ and $k$ be positive integers such that $n = dk + r$, where $0 \leq r < k$, and let $A_i(q)$ denote the $i^{th}$ Eulerian polynomial. 
	Also let $M_{n,k}$ denote the multinomial coefficient
	\[
		M_{n,k} =  {n \choose (d+1)^r,d^{k-r}}
	\]
	where $i^j$ indicates $i$ repeated $j$ times.
	We then have the identities
	\begin{eqnarray*}
		F_n^{\des_k}(q) &=& M_{n,k}A_{d+1}^r(q)A_d^{k-r}(q)\\
		F_n^{\inv_k}(q) &=&  M_{n,k}[d+1]_q^r![d]_q^{k-r}!
	\end{eqnarray*}
\end{thm}

\begin{proof}
	Let $k \in [n-1]$ and consider $\phi$ defined above.
	Note that $\phi$ is an $M_{n,k}$-to-one function since, given $(\sigma_1,\dots,\sigma_k) \in \symm_{d+1}^r \times \symm_d^{k-r}$, there are $M_{n,k}$ ways to partition 
	$[n]$ into the subsequences $\beta_{n,k}^i(\sigma)$ such that $\std \beta^i_{n,k}(\sigma) = \sigma_i$ for all $i$.
	Also note that
	\[
		\des_k(\sigma) = \sum_{i=1}^k \des(\std \beta^i_{n,k}(\sigma)).
	\]
	Thus,
	\begin{eqnarray*}
		F_n^{\inv_k}(q)&=&\sum_{\sigma\in\symm_n} q^{\des_k\sigma} \\
					&=& M_{n,k}\left(\sum_{(\sigma_1,\dots,\sigma_k)\in\symm_{d+1}^r \times \symm_d^{k-r}}  q^{\des\sigma_1}\dots q^{\des\sigma_k}\right)\\
					&=& M_{n,k}A_{d+1}^r(q)A_d^{k-r}(q).
	\end{eqnarray*}
	This proves the first identity.
		
	The second identity follows completely analogously, with the main modification being that an element of $\Inv_k(\sigma)$ corresponds to a usual inversion in some unique $\std \beta_{n,k}^j(\sigma)$.
\end{proof}

Given $\des_K$ and $\inv_K$, one must wonder what the corresponding generalizations of $\exc$ and $\maj$ are whose relationships with $\des_K$ and $\inv_K$ parallel that of the classical statistics.
To do this, we define the multiset
\[
	\Exc_K(\sigma) = \bigcup_{k \in K} \biguplus_{i = 1}^k \multiset{j \in [n-1] \mid \lceil j/k \rceil \in \Exc(\std \beta_{n,k}^i(\sigma))}
\]
and set $\exc_K(\sigma) = | \Exc_K(\sigma)|$, 
and also set
\[
	\maj_K(\sigma) = \sum_{k \in K} \sum_{i \in \Des_k(\sigma)} \left\lceil \frac{i}{k} \right\rceil.
\]
We could equivalently state that
\[
	\sum_{i=1}^k \maj(\std \beta^i_{n,k}(\sigma)) \text{ and } \exc_K(\sigma) =  \sum_{i=1}^k \exc(\std \beta^i_{n,k}(\sigma)).
\]
These are exactly the definitions needed in order to obtain identities that parallel $F_n^{\des}(q) = F_n^{\exc}(q)$ and $F_n^{\inv}(q) = F_n^{\maj}(q)$, as we will soon see.

One important distinct to make between $\exc$ and $\exc_K$ is the following.
If $\sigma = a_1a_2\dots a_n$ and $\tau = b_1b_2\dots b_n$, then even if $a_i = b_i$ for some $i$, one cannot say $i \in \Exc_K(\sigma)$ if and only if $i \in \Exc_K(\tau)$.
For example, if $\sigma = 4136572$, then $1 \in \Exc_2(\sigma)$, but if $\tau = 4153627$ then $1 \notin \Exc_2(\tau)$.

\begin{example}
	Again let $\sigma = 4136572$. 
	We then have
	\[
		\exc_{\{2,3\}}(\sigma) = |\{1,3,1,2\}| = 4
	\]
	and
	\[
		\maj_{\{2,3\}}(\sigma) = \left\lceil \frac{1}{2} \right\rceil + \left\lceil \frac{5}{2} \right\rceil + \left\lceil \frac{4}{3} \right\rceil = 6.
	\]
\end{example}

By constructing a nearly identical argument as in the proof Theorem~\ref{thm:main bijection}, and using the facts that $F_n^{\des}(q) = F_n^{\exc}(q)$ and $F_n^{\inv}(q) = F_n^{\maj}(q)$, we have the following corollary.

\begin{cor}
	The identities $F_n^{\des_k}(q) = F_n^{\exc_k}(q)$ and $F_n^{\inv_k}(q) = F_n^{\maj_k}(q)$ hold.
\end{cor}

Now that we have established the analogous parallels between the four classical statistics and their width-$k$ counterparts, we wish to explore what other structure is present.
A simple proposition relates $\des_k$ and $\inv_k$ when $k$ is large.

\begin{cor}
	For all $k \geq n/2$, $F_n^{\des_k}(q) = F_n^{\inv_k}(q)$.
\end{cor}

\begin{proof}
	Since $k \geq n/2$, the sets $\beta_{n,k}^i(\sigma)$ contain at most two elements.
	So, width-$k$ descents and width-$k$ inversions of $\sigma$ are identical.
\end{proof}

We now show that interesting behavior occurs when considering the function
\[
	G_{n,k}(q) = \sum_{\sigma \in \symm_n} q^{\des_k(\sigma) - \des_{n-k}(\sigma)}.
\]
According to computational data, the following conjecture holds for all $n \leq 9$ and $1 \leq k < n$ for which $\gcd(k,n) = 1$.

\begin{conj}
	If $\gcd(k,n) = 1$, then $G_{n,k}(q) = nq^{1-k}A_{n-1}(q)$.
\end{conj}

Several illustrative polynomials are given in Table~\ref{tab:difference data}.
This does not hold more generally, and it would be interesting to determine if a general formula exists.

\begin{table}\label{tab:difference data}
{\renewcommand{\arraystretch}{1.2}
\[	\begin{array}{c|ccc}
		k & G_{6,k} & G_{8,k}(q) & G_{9,k}(q) \\ \hline
		1 & 6A_5(q) & 8A_7(q) & 9A_9(q) \\
		2 & 180A_2(q)^2 & 1120A_3(q)^2 & 9q^{-1}A_9(q) \\
		3 & 6! & 8q^{-2}A_7(q) & 45360A_2(q)^3 \\
		4 & 180q^{-2}A_2(q)^2 & 8! & 9q^{-3}A_9(q) \\
		5 & 6q^{-4}A_5(q) & 8q^{-4}A_7(q) & 9q^{-4}A_9(q) \\
		6 & & 1120q^{-4}A_3(q)^2 & 45360q^{-3}A_2(q)^3 \\
		7 & & 8q^{-6}A_7(q) & 9q^{-6}A_9(q) \\
		8 & & & 9q^{-7}A_9(q) \\
	\end{array}
\]
}
\caption{The polynomials $G_{n,k}(q)$ for $n=6,8,9$ and $1 \leq k \leq n$.}
\end{table}

\begin{question}
	For which values of $n,k$ does there exist a closed formula for $G_{n,k}(q)$, and what is the formula?
\end{question}


\section{Pattern Avoidance}\label{sec:avoidance}

We say that a permutation $\sigma \in \symm_n$ \emph{contains the pattern} $\pi \in \symm_m$ if there exists a subsequence $\sigma'$ of $\sigma$ such that $\std(\sigma') = \pi$.
If no such subsequence exists, then we say that $\sigma$ \emph{avoids the pattern} $\pi$.
If $\Pi \subseteq \symm$, then we say $\sigma$ \emph{avoids} $\Pi$ if $\sigma$ avoids every element of $\Pi$.
The set of all permutations of $\symm_n$ avoiding $\Pi$ is denoted $\Av_n(\Pi)$.
In a mild abuse of notation, if $\Pi = \{\pi\}$, we will write $\Av_n(\pi)$.

In this section, we consider the functions
\[
	F_n^{\st}(\Pi;q) = \sum_{\sigma \in \Av_n(\Pi)} q^{\st \sigma},
\]
which specializes to $F_n^{\st}(q)$ if $\Pi = \emptyset$.
In most instances, $F_n^{\des_k}$ will be the main focus, but $F_n^{\des_K}$ and $F_n^{\inv_k}$ will also make appearances.

An important concept within pattern avoidance is that of Wilf equivalence.
Two sets $\Pi,\Pi' \subset \symm$ are said to be \emph{Wilf equivalent} if $|\Av_n(\Pi)| = |\Av_n(\Pi')|$ for all $n$.
In this case, we write $\Pi \equiv \Pi'$ to denote this Wilf equivalence, which is indeed an equivalence relation.
For example, it is known \cite{KnuthVol1,MacMahonCA} that whenever $\pi, \pi' \in \symm_3$, then $|\Av_n(\pi)| = |\Av_n(\pi')| = C_n$, where
\[
	C_n = \frac{1}{n+1}\binom{2n}{n}
\]
is the $n^{th}$ Catalan number.

Proving whether $\Pi \equiv \Pi'$ is often quite difficult, and their Wilf equivalence does not imply that $F_n^{\st}(\Pi;q) = F_n^{\st}(\Pi';q)$.
However, in some instances, the problems of establishing these identities have straightforward solutions by applying basic transformations on the elements of the avoidance classes.
Given $\pi = a_1\dots a_n \in \symm_m$, let $\pi^r$ denote its \emph{reversal} and $\pi^c$ denote its \emph{complement}, respectively defined by
\[
	\pi^r = a_ma_{m-1}\dots a_1 \text{ and } \pi^c = (m+1-a_1)(m+1-a_2)\dots (m+1-a_m).
\]
Similarly, given $\Pi \subseteq \symm$, we let
\[
		\Pi^r = \{\pi^c \mid \pi \in \Pi\} \text{ and } \Pi^c = \{\pi^r \mid \pi \in \Pi\}
\]
be the \emph{reversal} and \emph{complement} of $\Pi$, respectively.

Our results of this section begin with a multivariate generalization of $F_n^{\st}(\Pi;q)$, and with showing how its specializations describe relationships among $\Pi,\Pi^r$, and $\Pi^c$.

\begin{defn}
	Fix $\Pi \subseteq \symm$. Define
	\[
		T_n(\Pi; t_1,\ldots,t_{n-1}) = \sum_{\sigma \in \Av_n(\Pi)} \prod_{k = 1}^{n-1}t_k^{\des_k(\sigma)}.
	\]
\end{defn}

Several conclusions quickly follow. 
This function specializes to $F_n^{\des_K}(\Pi;q)$ by setting $t_i = q$ for $i \in K$ and $t_i = 1$ for all $i \notin K$.
We can also recover $F_n^{\inv_K}(\Pi;q)$ from $T_n$ by setting $t_i = q$ whenever $i \in [n-1] \cap k\ZZ$ for some $k \in K$, and setting $t_i = 1$ otherwise.
Additionally, when $\Pi = \{\pi\}$, a nice duality appears, providing a mild generalization of Lemma~2.1 from \cite{DokosEtAl}.

\begin{prop}\label{prop:comp reverse rotate}
	For any $\Pi \subseteq \symm$, let $\Pi'$ denote either $\Pi^r$ or $\Pi^c$.	
	We then have
	\[
		T_n(\Pi';t_1,\ldots,t_{n-1}) = t_1^{n-1}t_2^{n-2}\cdots t_{n-1}T_n(\Pi;t_1^{-1},\ldots,t_{n-1}^{-1}).
	\]
	Consequently,
	\[
		T_n((\Pi^r)^c;t_1,\ldots,t_{n-1}) = T_n(\Pi;t_1,\ldots,t_{n-1}).
	\]
\end{prop}

\begin{proof}
	It is enough to prove the claim when the set of patterns being avoided is $\{\pi\}$ for some $\pi \in \symm_m$, since the full result follows by applying the argument to all elements of $\Pi$ simultaneously.
	
	First consider when $\pi' = \pi^c$ and $\sigma \in \Av_n(\pi)$. 
	Because $\sigma \in \Av_n(\pi)$ if and only if $\sigma^c \in \Av_n(\pi^c)$,
	we have that for each $k$, $i \in \Des_k(\sigma)$ if and only if $i \notin \Des_k(\sigma^c)$.
	This implies $\Des_k(\sigma^c) = [n-k] \setminus \Des_k(\sigma)$, hence $\des_k(\sigma^c) = n-k-\des_k(\sigma)$.
	So,
	\begin{eqnarray*}
		T_n(\pi^c;t_1,\ldots,t_{n-1}) &=& \sum_{\sigma \in \Av_n(\pi^c)} \prod_{k = 1}^{n-1}t_k^{\des_k(\sigma)} \\
							  &=& \sum_{\sigma \in \Av_n(\pi)} \prod_{k = 1}^{n-1}t_k^{\des_k(\sigma^c)} \\
							  &=& \sum_{\sigma \in \Av_n(\pi)} \prod_{k = 1}^{n-1}t_k^{n-k-\des_k(\sigma)} \\
							  &=& t_1^{n-1}t_2^{n-2}\cdots t_{n-1}T_n(\pi;t_1^{-1},\ldots,t_{n-1}^{-1}).
	\end{eqnarray*}
	Proving that the result holds for $\pi' = \pi^r$ follows similarly.
	
	The second identity in the proposition statement holds by applying the first identity twice: first for $\Pi^r$ and then for $\Pi^c$. 
\end{proof}

The above identities significantly reduce the amount of work needed to study $F_n^{\des_K}(\Pi;q)$ for all $\Pi \subseteq \symm_n$.
For the remainder of this paper, we systematically approach $\Pi$ for $|\Pi| \leq 2$.


\subsection{Avoiding singletons}

By Proposition~\ref{prop:comp reverse rotate}, we immediately get
\[
	F_n^{\des_k}(123;q) = q^{n-k}F_n^{\des_k}(321;q^{-1})
\]
and
\[
	F_n^{\des_k}(132;q) = F_n^{\des_k}(213;q) = q^{n-k}F_n^{\des_k}(231;q^{-1}) = q^{n-k}F_n^{\des_k}(312;q^{-1}).
\]
So, studying $F_n^{\des_k}(\pi;q)$ for $\pi \in \symm_3$ reduces to studying the function for a choice of one pattern from $\{123,321\}$ and one from the remaining patterns.
For some choices of $\Pi$, the permutations in $\Av_n(\Pi)$ are especially highly structured, which leads to similar arguments throughout the rest of this paper.

We begin with $\pi = 312$. 
Notice that if $a_1\dots a_n \in \Av_n(312)$ and $a_i = 1$, then $\std(a_1\dots a_{i-1}) \in \Av_{i-1}(312)$, $\std(a_{i+1}\dots a_n) \in \Av_{n-i}(312)$, and 
\[
	\max\{a_1,\ldots,a_{i-1}\} < \min\{a_{i+1},\ldots,a_n\}.
\]

\begin{prop}
	For all $n$,
	\[
		F_n^{\des_k}(312;q) = \sum_{i=1}^k C_{i-1}F_{n-i}^{\des_k}(312;q) + \sum_{i=k+1}^n qF_{i-1}^{\des_k}(312;q)F_{n-i}^{\des_k}(312;q)
	\]
	where $C_i$ is the $i^{th}$ Catalan number.
\end{prop}

\begin{proof}
	First consider when $\sigma = a_1\dots a_n \in \Av_n(312)$ and $a_i = 1$ for some $i \leq k$.
	By the discussion preceding this proposition, $j \notin \Des_k(\sigma)$ for any $j \leq i$.
	So, none of the $C_{i-1}$ possible permutations that make up $\std(a_1\dots a_{i-1})$ contribute to $\des_k(\sigma)$.
	The only contributions to $\des_k(\sigma)$ come from $\std(a_{i+1}\dots a_n) \in \Av_{n-i}(312;q)$.
	The overall contribution to $F_n^{\des_k}(312;q)$ is the first summand of the identity.
	
	Now suppose $a_i = 1$ for some $i > k$.
	Each choice of $a_1\dots a_i$, contributes to $\des_k(\sigma)$ as usual, but there will be an additional width-$k$ descent produced at $i-k$.
	The elements $a_{i+1}\dots a_n$ contribute to $\des_k(\sigma)$ as before. 
	The overall contribution to $F_n^{\des_k}(312;q)$ is the first summand of the identity.
	Since we have considered all possible indices $i$ for which we could have $a_i = 1$, we add the two cases together and are done.
\end{proof}

Note that when we set $q=1$, the above recursion specializes to the well-known recursion for Catalan numbers $C_{n+1} = \sum_{i=0}^n C_iC_{n-i}$.

We can use the previous proposition in conjunction with Proposition~\ref{prop:comp reverse rotate} to produce formulae for $F_n^{\des_k}(\Pi;q)$ and $F_n^{\inv_k}(\Pi;q)$ whenever $\Pi$ is a single element of $\symm_3$ other than $123$ or $321$. 
The only nontrivial work required, then, is to compute the degrees of the two polynomials.

\begin{cor}
	For all $n$, 
	\[
		\deg F_n^{\des_k}(312;q) = n-k \text{ and } \deg F_n^{\inv_k}(312;q) = \sum_{i=1}^{n-k} \left\lfloor \frac{n-i}{k} \right\rfloor.
	\]
\end{cor}

\begin{proof}
	The degrees of the above polynomials are given by identifying a permutation in $\Av_n(312)$ with the most possible descents.
	This is satisfied by $n(n-1)\dots 21 \in \Av_n(312)$ which has $n-k$ descents of width $k$.
	Determining the number of width-$k$ inversions in this permutation is done similarly.
\end{proof}

Although $321$ is Wilf equivalent to $312$, it is not so obvious how to construct a recurrence relation for $F_n^{\des_k}(321;q)$.
To help describe the elements of $\Av_n(321)$, first let $\sigma = a_1\dots a_n \in \symm$ and call $a_i$ a \emph{left-right maximum} if $a_i > a_j$ for all $j < i$.
Thus $\sigma \in \Av_n(321)$ if and only if its set of non-left-right maxima form an increasing subsequence of $\sigma$. 
Indeed, if the non-left-right maxima did contain a descent, then there would be some left-right maximum preceding both elements, which violates the condition that $\sigma$ avoid $321$.
Despite such a description, using it to reveal $F_n^{\des_k}(321;q)$ has thus far been unsuccessful.
So, we leave the following as an open question.

\begin{question}
	Is there a closed formula or simple recursion for $F_n^{\des_k}(321;q)$ (equivalently, for $F_n^{\des_k}(123;q)$)?
\end{question}


\subsection{Avoiding doubletons}

At this point, we begin studying the functions $F_n^{\des_k}(\Pi;q)$ when avoiding doubletons from $\symm_3$.
Recall that, by the Erd\H{o}s-Szekeres theorem \cite{ErdosSzekeres}, there are no permutations in $\symm_n$ for $n \geq 5$ that avoid both $123$ and $321$.
Thus, we will not consider this pair.
Additionally, the functions $F_n^{\inv_k}(\Pi;q)$ are quite unwieldy, so we will not consider these either.

For our first nontrivial example, we begin with $\{123,132\}$.
Permutations $a_1\dots a_n \in \Av_n(123,132)$ have the following structure: 
for any $i = 1,\dots, n$, if $a_i = n$, then the substring $a_1a_2\dots a_{i-1}$ is decreasing 
and consists of the elements from the interval $[n-i+1,n-1]$.
Additionally, the substring $a_{i+1}\dots a_n$ is an element of $\Av_{n-i}(123,132)$.
This structure makes it easy to show that there are $2^{n-1}$ elements of $\Av_n(123,132)$ \cite{SimionSchmidt}.

\begin{prop}\label{prop: 123 132}
	For all $n$ and $1 \leq k \leq n-1$,
	\begin{eqnarray*}
		F_n^{\des_k}(123,132;q) &=& \sum_{i=1}^k q^{\min(i,n-k)}F_{n-i}^{\des_k}(123,132;q) \\
							&& + \sum_{i=k+1}^{n-k} q^{\min(i-1,n-k-1)}F_{n-i}^{\des_k}(123,132;q) \\
							&& + 2^{n-\max(k+1,n-k+1)}q^{n-k-1}.
	\end{eqnarray*}
\end{prop}

\begin{proof}
	Let $\sigma = a_1\dots a_n \in \Av_n(123,132;q)$.
	If $a_i = n$ for $i \leq k$, then it is clear from the preceding description of the elements in $\Av_n(123,132)$ that all of $1,\dots,\min(i,n-k)$ are elements of $\Des_k(\sigma)$.
	If $j > n-k$ for some $j$, then $j \notin \Des_k(\sigma)$ since $a_{j+k}$ does not exist. 
	The remaining elements of $\Des_k(\sigma)$ arise as a width-$k$ descent of $a_{i+1}\dots a_n$, which accounts for the factor of $F_{n-i}^{\des_k}(123,132;q)$ in the first summand.
	
	Now, if $k+1 \leq n-k$ and $a_i = n$ for $i = k+1,\dots,n-k$, then every element of $1,\ldots,i$ except $i-k$ is an element of $\Des_k(\sigma)$.
	If $k+1 > n-k$, then the second summand is empty and nothing is lost by continuing to the case of $a_i = n$ for $i \geq \max(k+1,n-k+1)$.
	Again, the remaining elements of $\Des_k(\sigma)$ are the width-$k$ descents of $a_{i+1}\dots a_n$, hence the additional factor of $F_{n-i}^{\des_k}(123,132;q)$.
	This accounts for the second summand.
	
	Finally, let $m = \max(k+1,n-k+1)$.
	If $a_i = n$ for $i \geq m$, then all of $1,\ldots,n-k$ are descents except $i-k$, and these are the only possible width-$k$ descents.
	In particular, none of $i+1,i+2,\dots ,n$ can be the index for a width-$k$ descent.
	This leads to the sum
	\[
		\sum_{i = m}^n q^{n-k-1}|\Av_{n-i}(123,132)| = \left(1 + \sum_{i = m}^{n-1} 2^{n-i-1}\right)q^{n-k-1} = 2^{n-m}q^{n-k-1},
	\]
	which accounts for the third summand.
\end{proof}

Next, we consider $\{123,312\}$.
We proceed similarly as before but with some minor differences, reflecting the new structure we encounter.
If $\sigma = a_1\dots a_n \in \Av_n(123,312)$ and $a_i = 1$ for some $i < n$, then
$\sigma$ is of the form
\[
	\sigma = i(i-1)\dots 21n(n-1)\dots (i+2)(i+1),
\]
since neither subsequence $a_1\dots a_{i-1}$ and $a_{i+1}\dots a_n$ may contain an ascent.
If $a_n = 1$, though, then $\std(a_1\dots a_{n-1}) \in \Av_{n-1}(123,312)$.

\begin{prop}
	For all $n$ and $1 \leq k < n$,
	\begin{eqnarray*}
		F_n^{\des_k}(123,312;q) &=& \sum_{i=1}^k q^{\max(0,n-k-i)} + \sum_{i=k+1}^{n-1} q^{\max(n-2k,i-k)} \\
							&& + qF_{n-1}^{\des_k}(123,312;q)
	\end{eqnarray*}
\end{prop}

\begin{proof}
	Suppose $a_i = 1$ for some $i \leq k$.
	Using the description of elements in $\Av_n(123,312)$, we know there are $n-k-i$ width-$k$ descents.
	Since there is only one such permutation for each $i$, we simply add all of the $q^{n-k-i}$, so long as $n-k-i \geq 0$.
	If this inequality does not hold, then there are no width-$k$ descents in this range.
	This accounts for the first summand.
	
	The second summand is computed very similarly to that of the second summand in Proposition~\ref{prop: 123 132}.
	The final summand is a direct result of noting that when $a_n = 1$, then $(a_1-1)\dots (a_{n-1}-1)$ may be any element of $\Av_{n-1}(123,312;q)$, which accounts for the factor $F_{n-1}^{\des_k}(123,312;q)$.
	For each of these choices, we know $n-k \in \Des_k(\sigma)$, hence the factor of $q$.
	Adding the sums results in the identity claimed.
\end{proof}

Now we consider $\{132,231\}$.
Note that if $a_1\dots a_n \in \Av_n(132,231)$, then $a_i \neq n$ for any $1 < i < n$.
If $a_1 = n$, then $\std(a_2\dots a_n) \in \Av_{n-1}(132,231)$, and similarly if $a_n = n$.
Once again, this allows us to quickly compute that $|\Av_n(132,231)| = 2^{n-1}$.

\begin{prop}\label{prop: 132 231}
	Let $K = \{k_1,\ldots,k_l\}$ be a nonempty subset of $[n-1]$ whose elements are listed in increasing order.
	We then have
	\[
		F_n^{\des_K}(132,231;q) = \prod_{i=1}^{l+1} (1+q^{i-1})^{k_i - k_{i-1}}
	\]
	where $k_0 = 1$ and $k_{l+1} = n$.
\end{prop}

\begin{proof}
	It follows from the description of elements in $\Av_n(132,231)$ that there are two summands in a recurrence for $F_n^{\des_K}(132,231;q)$: one corresponding to $a_1 = n$ and one corresponding to $a_n = n$.
	When $a_1 = n$, then there are $|K| = l$ copies of $1 \in \Des_K(\sigma)$; if $a_n = n$, then $a_n$ makes no contribution to $\Des_K(\sigma)$.
	Thus, by deleting $n$ from $\sigma$, we get the recurrence
	\[
		F_n^{\des_K}(132,231;q) = (1+q^l)F_{n-1}^{\des_K}(132,231;q).
	\]
	Repeating this, a factor of $1+q^l$ appears until we get to
	\[
		F_n^{\des_K}(132,231;q) = (1+q^l)^{n-k_l}F_{k_l}^{\des_K}(132,231;q).
	\]
	At this point, note that 
	\[
		F_{k_l}^{\des_K}(132,231;q) = F_{k_l}^{\des_{K \setminus \{k_l\}}}(132,231;q).
	\]
	Repeating the previous argument ends up with the identity claimed.	
\end{proof}

Next, consider $\{132,213\}$.
If $\sigma = a_1\dots a_n \in \Av_n(132,213)$, then if $a_i = n$ for any $i$, then $a_1\dots a_{i-1}$ must be increasing in order to avoid $213$.
Moreover,
\[
	\max(a_{i+1},\dots,a_n) < a_1
\]
in order to avoid $132$, and $\std(a_{i+1}\dots a_n) \in \Av_{n-i}(132,213)$.

\begin{prop}
	For all $n$ and $1 \leq k < n$,
	\begin{eqnarray*}
		F_n^{\des_k}(132,213;q) &=& \sum_{i=1}^k q^{\min(i,n-k)}F_{n-i}^{\des_k}(132,213;q) \\
						&& + \sum_{i = k+1}^{n-k} q^{\min(k,n-i)}F_{n-i}^{\des_k}(132,213;q) \\
						&& + \sum_{i = \max(k+1,n-k+1)}^n q^{n-i}F_{n-i}^{\des_k}(132,213;q).
	\end{eqnarray*}
\end{prop}

\begin{proof}
	The proof of this is entirely analogous to the proof of Proposition~\ref{prop: 123 132}.
\end{proof}

Next, we consider $\{132,312\}$.
For each $i = 1,\ldots, n-1$, either $a_{i+1} = \max\{a_1,\ldots,a_i\} + 1$ or $a_{i+1} = \min\{a_1,\ldots,a_i\} - 1$.
This doubleton often results in especially pleasant formulas, and our results are no exception.

\begin{prop}\label{prop: 132 312}
	Let $K = \{k_1,\ldots,k_l\}$ be a nonempty subset of $[n-1]$ whose elements are listed in increasing order.
	We then have
	\[
		F_n^{\des_K}(132,312;q) = \prod_{i=1}^{l+1} (1+q^{i-1})^{k_i - k_{i-1}}
	\]
	where $k_0 = 1$ and $k_{l+1} = n$.
\end{prop}

\begin{proof}
	From the description of elements $\sigma = a_1\dots a_n \in \Av_n(132,312)$, we know that either $a_n = 1$ or $a_n = n$.
	In the former case, $n-k \in \Des_K(\sigma)$ for each $k \in K$, and in the latter case, $n-k \notin \Des_K(\sigma)$ for each $k \in K$.
	This leads to the recurrence
	\[
		F_n^{\des_K}(132,312;q) = (1+q^l)F_{n-1}^{\des_K}(132,312;q).
	\]
	Following an analogous argument as in the proof of Proposition~\ref{prop: 132 231} obtains the result.
\end{proof}

From the general formula given above, we can quickly determine $F_n^{\inv_k}(132,312;q)$.

\begin{cor}
	For fixed $n,k$, write $n = dk + r$ for unique nonnegative integers $d,r$ such that $0 \leq r < k$.
	We then have
	\[
		F_n^{\inv_k}(132,312;q) = F_n^{\inv_k}(132,231;q) = 2^{k-1}(1+q^d)^r\prod_{i = 1}^{d-1} (1+q^i)^k.
	\]
\end{cor}

\begin{proof}
	This is a direct consequence of the general formulas from Propositions~\ref{prop: 132 231} and \ref{prop: 132 312}, and recalling that
	\[
		F_n^{\inv_k}(\Pi;q) = F_n^{\des_{[n-1] \cap k\ZZ}}(\Pi;q).
	\]
\end{proof}

\subsection*{Acknowledgement} The author thanks Bruce Sagan for his many helpful comments.

\bibliographystyle{plain}
\bibliography{references}

\end{document}